\documentclass{article}

\usepackage[english]{babel}
\usepackage[utf8]{inputenc}
\usepackage[T1]{fontenc}
\usepackage{lmodern}
\usepackage{graphicx}
\usepackage{amsfonts}
\usepackage[nice]{nicefrac}
\usepackage{amsmath,amsthm,amssymb} 
\usepackage{thmtools}
\usepackage{caption}
\usepackage{array}
\usepackage{color}
\usepackage{url}
\usepackage[top=2cm,bottom=3cm,right=2cm,left=2cm]{geometry}
\usepackage{enumerate}
\usepackage{hyperref}
\usepackage{mathrsfs}
\usepackage{bbold}

\newcommand{\ssup}[1] {{\scriptscriptstyle{({#1}})}}
\newcommand{\sfrac}[2] {\mbox{$\frac{#1}{#2}$}}

\usepackage{subfig}
\renewcommand{\c}{\mathcal}
\renewcommand{\b}{\mathbb}

\font\dsrom=dsrom10 scaled 1200
\def \indi{\textrm{\dsrom{1}}}

\newcommand{\Var}{\mathtt{Var}}

\newcommand{\eps}{\varepsilon}
\newcommand{\1}{1\hspace{-0.098cm}\mathrm{l}}
\newcommand{\N}{\mathbb N}
\renewcommand{\P}{\mathbb P}
\newcommand{\E}{\mathbb E}

\newcommand{\cF}{\mathscr F}

\newtheorem{remark}{Remark}

\declaretheorem[thmbox=S, numberwithin=section, name=Theorem]{thm}
\newtheorem{prop}[thm]{Proposition}

\newtheorem{lem}[thm]{Lemma}
\newtheorem{cj}[thm]{Conjecture}
\newtheorem{cor}[thm]{Corollary}

\makeatletter 
\renewenvironment{proof}[1][Proof] {\par\pushQED{\qed}\normalfont\topsep6\p@\@plus6\p@\relax\trivlist\item[\hskip\labelsep\bfseries#1\@addpunct{.}]\ignorespaces}{\popQED\endtrivlist\@endpefalse} 
\makeatother

\usepackage{parskip}

\title{Non-extensive condensation in reinforced branching processes}
\author{Steffen Dereich, C{\'e}cile Mailler and Peter M{\"o}rters}
\date{}

\begin{document}

\thispagestyle{empty}
\maketitle

\begin{abstract}
\noindent 
We study a class of branching processes in which a population consists of immortal
individuals equipped with a fitness value. Individuals produce offspring with a rate given by their fitness,
and offspring may either belong to the same family, sharing the fitness of their parent, or be founders of new families, with a fitness sampled from a fitness distribution~$\mu$.  Examples that can be embedded in this class are stochastic house-of-cards models, urn models with reinforcement, and the preferential attachment tree of Bianconi and Barab\'asi. Our focus is on the case when the fitness distribution $\mu$ has bounded support and regularly varying tail 
at the essential supremum. In this case there exists a condensation phase, in which asymptotically a proportion of mass  in the empirical fitness distribution of the overall population condenses in the maximal fitness value.  Our main results describe the asymptotic behaviour of the size and fitness of the largest family at a given time. In particular, we show that as time goes to infinity the size of the largest family is always negligible compared to the overall population size. This implies that condensation, when it arises, is non-extensive and emerges as a collective effort of several families none of which  can create a condensate on its own.  Our result disproves claims made in the physics literature in the context of preferential attachment trees.
\end{abstract}


\tableofcontents

\newpage

\section{Background and motivation}

The principal aim of this paper is to study the emergence of a condensate in stochastic models. For this purpose we consider a class of branching processes with reinforcement, which  probably constitute the  easiest class of models, in which this question can be studied in a meaningful way. Still we shall see that, due to the reinforcement, these models display rather complex behaviour and not all relevant questions on their behaviour will be answered.
\smallskip

Although our models can describe a variety of objects, see the examples below, we shall describe them as a structured population. 
Parameters of our model are a fitness distribution~$\mu$ on the positive reals,  and positive numbers $\beta, \gamma\leq1$ with $\beta+\gamma\geq 1$.
At any time~$t$ the population consists of a finite number~$N(t)$ of individuals. 
Each individual in the population has a fitness, and
individuals are organised into families, 
such that all members of a family have the same fitness. 
%
The process is started with one family of one individual, whose fitness is drawn from the distribution~$\mu$. 
Suppose, at time $t\geq 0$, the population consists of $M(t)$ families, 
and there are $Z_n(t)$ individuals of fitness $F_n$ in the $n$th family, for $1\leq n \leq M(t)$. 
Independently in every family birth events occur with a time-dependent rate~$F_n Z_n(t)$. 
When a birth event occurs in the $n$th family, independently of everything else, 
one or both of the following happen,
\begin{itemize}
\item with probability $\beta$ a new family is founded, 
initially consisting of one individual equipped with a fitness drawn,
independently of everything else,  from the distribution~$\mu$;
\item with probability $\gamma$ a new individual with fitness~$F_n$ is added to the $n$th family. The ability of the system to reproduce particles of the same type constitutes the reinforcement, see~\cite{Pemantle}.
\end{itemize}
Note that both things happen simultaneously with probability $\beta+\gamma-1\geq0$. 
If $\mu$ has all exponential moments the total number~$N(t)$ of individuals in the population remains finite at all times, see for example Corollary~3.3 in~\cite{Komjathy}.  
Our main object of interest is the \emph{empirical fitness distribution} at time~$t$, which is defined as
\begin{equation}\label{empfit}
\Xi_t=\frac1{N(t)} \sum_{n=1}^{M(t)} Z_n(t)\,\delta_{F_n}.
\end{equation}

In this paper we focus on bounded fitness distributions~$\mu$, and specifically the case in which a condensation phenomenon occurs, which we describe in some detail.  Different phenomena occur in the case of unbounded fitness distributions and these will be investigated in a companion paper~\cite{DMM16}. From now on we assume that  $\mu$ is a probability measure supported by a bounded subinterval of the positive reals. Without loss of generality we assume 
that $\mu$ has essential supremum equal to one. To avoid degeneracies we also assume that $\mu$ has no atom at one.
\medskip

We now describe our three main examples motivating our work.

{\bf Example~1:} \emph{Branching process with selection and mutation.}

This model is a stochastic \emph{house-of-cards model} in a similar vein as Kingman's model 
(which is deterministic and much easier to analyse, see~\cite{Kingman, DM13}). 
We start with a single individual with a genetic fitness chosen according to~$\mu$. 
Individuals never die and give birth to new individuals with a rate equal to their genetic fitness, 
the different reproduction rates causing the selection effect. 
When a new individual is born it is a mutant with probability~$\beta$, 
in which case it gets a fitness drawn independently of everything else 
from~$\mu$. If the new individual is not a mutant, it inherits the fitness of its parent. 
The model corresponds to the parameter choice $\gamma=1-\beta$ in our framework. 
Observe that a mutation causes the complete loss of genetic information in the 
affected individual's ancestry, pictorially speaking `the genetic house of cards collapses'. 
This is why the term {house-of-cards model} is used for this process, see~\cite{Hodgkins}
for a discussion of the relevance of these models in the theory of evolution.
\smallskip

The number of families $M(t)$ corresponds to the number of mutants in the population at time~$t$.
We can describe the process $(M(t))_{t>0}$ as a Crump-Mode-Jagers process, using the framework of~\cite{Nerman}. A mutant $x$ born at time $\tau$ with fitness~$f$ produces new mutants at ages according to a random point process~$\xi$. This process is a Cox process, i.e.\ a Poisson process with a random intensity measure~$\beta f \phi_x(s) \, ds$. The function~$\phi_x(s)$ is given by the size 
at time $\tau+s$
of the family founded by the mutant and is therefore a Yule processes with intensity $(1-\beta)f$. 

The key assumption for the convergence theory of Crump-Mode-Jagers processes is the existence of a \emph{Malthusian parameter}, i.e.\ an $\alpha>0$ such that
$$1=\int_0^\infty e^{-\alpha s}  \, {\mathbb E} \xi(ds).$$
In our case we have, for $\alpha\geq1-\beta$, that
$$\int_0^\infty e^{-\alpha s}  \, {\mathbb E} \xi(ds)=
{\mathbb E} \int_0^\infty e^{-\alpha s} \beta f \phi_x(s) \, ds
=\beta \int f \int_0^\infty  e^{-\alpha s+(1-\beta)f s}  \, ds \, \mu(df)
=\beta \int \frac{f}{\alpha -(1-\beta)f} \, \mu(df) .$$
Hence a Malthusian parameter exists if and only if \smash{$\frac\beta{1-\beta} \int \frac{f}{1-f} \, \mu(df)\geq 1$. }
If this condition fails, the classical convergence theory of Crump-Mode-Jagers processes fails and very little is known about this case. 
In particular, in our model the precise asymptotics of $M(t)$ is unknown. We show that in this case a phenomenon of condensation occurs, which loosely 
speaking means  that a positive proportion of individuals have fitnesses converging to the maximal possible value. Key questions motivating this project are: \emph{How fast is this convergence, when did the mutations arise that form the condensate,
and how many mutations contribute to the condensate?} 

\pagebreak[3]
\bigskip
 
{\bf Example~2:} \emph{Preferential attachment tree of Bianconi and Barab\'asi.}

This model is originally a discrete time network model. 
Putting it into our framework means embedding it into continuous time, a technique heavily advocated by Janson~\cite{Ja04}, 
who attributes the method to Athreya and Karlin~\cite{AK68}, and by Bhamidi~\cite{Bh07}.
The network is constructed successively, starting with one vertex which is formally given degree one. 
The vertex is given a fitness, randomly chosen according to~$\mu$. 
At every time step a new vertex is introduced, equipped with a fitness, randomly chosen according to~$\mu$, 
and linked to one of the existing vertices. 
The probability of an existing vertex being  chosen is proportional to the product of 
its fitness and its degree at the time when the new vertex is introduced. 
As new vertices prefer to attach to existing vertices of high degree and high fitness, 
this is called a \emph{preferential attachment} model. 

In our representation we choose $\beta=\gamma=1$ and observe the system at the birth times of 
individuals.  We think of every family as a vertex in the network, and of the size of a family as its degree. 
Note that when the $n$th birth event takes place, 
it arises in each of the existing families with a probability 
proportional to the product of its fitness and its degree. 
At the birth event a new family is founded, i.e.\ a new vertex is introduced, 
and simultaneously the family that has given birth is increased in size by one, 
meaning that the degree of the corresponding vertex is incremented by one. 
Our representation only keeps track of the vertices and their degrees, 
not of the actual edges. 
But this does not matter as the main object of interest for us is 
the long-term behaviour of the degree-weighted fitness distribution, 
which coincides with the empirical fitness distribution in our framework.

This model was analysed by Borgs et al.~\cite{BC07} who proved the existence of a \emph{innovation-pays-off phase} 
in which a proportion of the mass in the degree-weighted fitness distribution condenses in the maximal fitness. 
This behaviour was already predicted in~\cite{BB01} who called this phase the winner-takes-all phase, 
a heavily misleading name as we shall see below. 
The result is reproved in our  Theorem~\ref{th:cond}. 
Borgs et al.~\cite{BC07} state as an open problem 
\emph{`to give an exact quantitative description of the innovation-pays-off phase. [...] 
How are the links distributed among the highest fitnesses present in the system at any given time? 
At what rate are new nodes with higher fitness taking over?'} 
Our main aim here is to make progress on this problem.
\bigskip

{\bf Example 3:} \emph{Generalised P{\'o}lya urns.}

A class of generalised P{\'o}lya urns also falls into our framework, with general parameters $\beta, \gamma>0$ and $\mu$ as above. It can be described as an urn containing balls of different colours. Every colour has a given \emph{activity} chosen independently according to $\mu$. At time zero, the urn contains one ball of colour $1$. At every time step, a ball is drawn at random from the urn with probability proportional to its activity.
Then, the drawn ball is put back into the urn together with one or two new balls, at most one ball of the same and one of a new colour. 
A ball with the same colour is chosen with probability $\gamma$,  and a ball of a new colour with probability $\beta$. New colours are chosen independently according to $\mu$. To embed the urn model into our framework we again look at the times of birth events. Observe that $\Xi_t$ is now the empirical distribution of activities in the urn at time $t$.
\pagebreak[3]

Such generalised P{\'o}lya urns have apparently not been studied so far in full generality. 
Janson~\cite{Ja04} is looking at the case where $\mu$ is finitely supported, in which the  
condensation phenomenon, which is of interest to us, cannot arise. A related model has been studied by Chung et al.~\cite{CHJ03} 
who draw balls  depending in a non-linear way on the distribution of colours in the urn, and by 
Collevecchio et al.~\cite{CCL13} who allow for a time-dependent replacement rule. Their main focus 
is on the question whether there can be an unbounded number of balls of more than one colour,
and if not which colour eventually dominates. In our setup all colours will have an unbounded number 
of balls and we show that \emph{the asymptotic proportion of balls of any colour goes to zero}
uniformly as time goes to infinity.
\bigskip

\section{Statement of the results}
The reinforced branching process is described by the following family of random variables. 
We denote by
\begin{itemize}
\item $N(t)$ the total size of the population at time $t$,
\item $M(t)$ the number of different families at time $t$, 
\item $\sigma_n$ the time of the $n^{\text{th}}$ birth event,
\item $\tau_n$ the time of the foundation of the $n^{\text{th}}$ family,
\item $Z_n(t)$ the size of the $n^{\text{th}}$ family at time $t$ (if $n>M(t)$ we set $Z_n(t)=0$), and
\item $F_n$ the fitness of the $n^{\text{th}}$ family.
\end{itemize}

We are first interested in the empirical fitness distribution $\Xi_t$ at time $t$, defined in \eqref{empfit}.
The asymptotic behaviour of this empirical distribution shows a phase transition 
between a fluid phase and a condensation phase.
The condition for condensation is
\begin{equation}\label{cond}\tag{Cond}
\frac{\beta}{\beta+\gamma} \int_0^1 \frac{1}{1-x} \, d\mu(x)  <  1
\quad\mbox{ or, equivalently, }\quad \frac{\beta}{\gamma}\int_0^1 \frac{x}{1-x}  \, d\mu(x)<  1,
\end{equation}
as stated in the following theorem. 
\bigskip

\begin{thm}[Existence of a condensation phase]\label{th:cond}
If \eqref{cond} fails, then there exists a unique 
$\lambda^{\star}\in[\gamma,\beta + \gamma)$ such that
\[\frac{\beta}{ \beta+\gamma} \int_0^1 \frac{\lambda^{\star}}{\lambda^{\star}-\gamma x} \, d\mu(x) =1,\]
otherwise let $\lambda^{\star}:=\gamma$. In both cases
\begin{itemize}
\item the empirical mean fitness 
$\int_0^1 x \, \Xi_t(dx)$ converges almost surely to $\nicefrac{\lambda^{\star}}{\beta+\gamma}$,
\item  and there exists a probability measure $\pi$ such that, almost
surely, the empirical fitness distribution $\Xi_t$ converges 
weakly to $\pi$.
\end{itemize}
The limit measure~$\pi$ of the empirical fitness distribution is given
\begin{itemize}
\item[(a)] if (\ref{cond}) fails
by $$ d\pi(x)= \frac{\beta}{\beta+\gamma} \frac{\lambda^\star}{\lambda^{\star}-\gamma x} d\mu(x).$$
\end{itemize}
\begin{itemize}
\item[(b)] if (\ref{cond}) is true by
$$d\pi(x)={\frac{\beta}{\beta + \gamma}} \frac{1}{1- x} d\mu(x) + \omega(\beta,\gamma) \delta_1(dx),$$
where
$$\omega(\beta,\gamma):=1 -\frac{\beta}{\beta+\gamma}\int_0^1 \frac{1}{1-x} d\mu(x)>0.$$
\end{itemize}
\end{thm}

\medskip
\pagebreak[3]

\begin{remark}
{\rm It is easy to see from the law of large numbers that
$$\frac{M(t)}{N(t)} \longrightarrow \frac{\beta}{\beta+\gamma} \qquad \mbox{ almost surely.}$$
Hence it is equivalent to ask for the absolute growth of either of the processes $(M(t)\colon t>0)$ or $(N(t)\colon t>0)$.  Given the population at time~$\sigma_n$ of the $n$th birth event, the waiting time $\sigma_{n+1}-\sigma_n$ until the next individual is born is exponentially distributed with rate 
\smash{$N(\sigma_n) \int x \, \Xi_{\sigma_n}(dx)\sim n \lambda^\star$}, where we have used that
$N(\sigma_n) \sim (\beta+\gamma)n$ by the law of large numbers. Hence
\smash{$\sigma_n \sim \frac1{\lambda^\star} \log n$}
and, in particular, we obtain, almost surely, 
$$\lim_{t\uparrow\infty} \frac1t \log N(t) = \lambda^\star.$$
If there is no condensation we can improve this to convergence of $e^{-t \lambda^\star} N(t)$  to a positive random variable, using the arguments sketched in Section~3.2 below.
But fine results about  the growth of the population in the condensation phase are hard to obtain, see also Section~8 .} 
\end{remark}
\medskip

\begin{remark}
{\rm We denote the part of the limit mass~$\pi$ which is absolutely continuous with respect to $\mu$ as \emph{bulk} and the part concentrated in the maximal fitness as \emph{condensate}.  The theorem shows 
that in the condensation phase, i.e.\ if \eqref{cond} holds, we are seeing a phenomenon of \emph{self-organised criticality}, 
as the number of individuals in the bulk and in the condensate are always kept on the same order of magnitude, without any tuning of parameters. In Dereich~\cite{De14} one can see that for a model without self-organisation it can be rather complicated to tune the parameters in such a way that one has coexistence of bulk and condensate.} 
\end{remark}
\medskip

Our interest in this paper lies in the \emph{emergence} of the condensate, i.e.\ how the condensate manifests itself at large finite times. Following the discussion of Bose-Einstein condensation in van den Berg et al.~\cite{BLP86} two alternative scenarios are possible:
\begin{itemize}
\item For the largest family, the proportion of individuals belonging to this family in the overall population at time~$t$  is asymptotically positive. This phenomenon of \emph{macroscopic occupancy} arises in condensation of the free Bose gas below a critical temperature, see~\cite{BLP86}.
\item No individual family makes an asymptotically  positive contribution. Instead, it is a collective effort of a growing number of families to form the condensate. This phenomenon is called   \emph{non-extensive condensation}. 
van den Berg et al.~\cite{BLP86}~have shown that this occurs in the free Bose gas for an intermediate temperature range.
\end{itemize}
\smallskip

We shall see in Theorem~\ref{nowinner} that in our model under a natural assumption on $\mu$ the second scenario prevails. To show this 
we need to investigate the behaviour of the largest family in our system. This requires some regularity assumptions on~$\mu$.  We henceforth assume that the fitness distribution $\mu$ has a regularly varying tail in one, meaning that there exists $\alpha>1$ and a slowly varying function $\ell$ with
\begin{align}\label{as}\tag{RV}
\mu(1-\eps,1)=\eps^\alpha \ell(\eps).
\end{align}
This corresponds to the most common type of behaviour of $\mu$ at its tip that allows a condensation phase.
\smallskip

\begin{figure}[ht]
\begin{center}
\includegraphics[width=.7\textwidth]{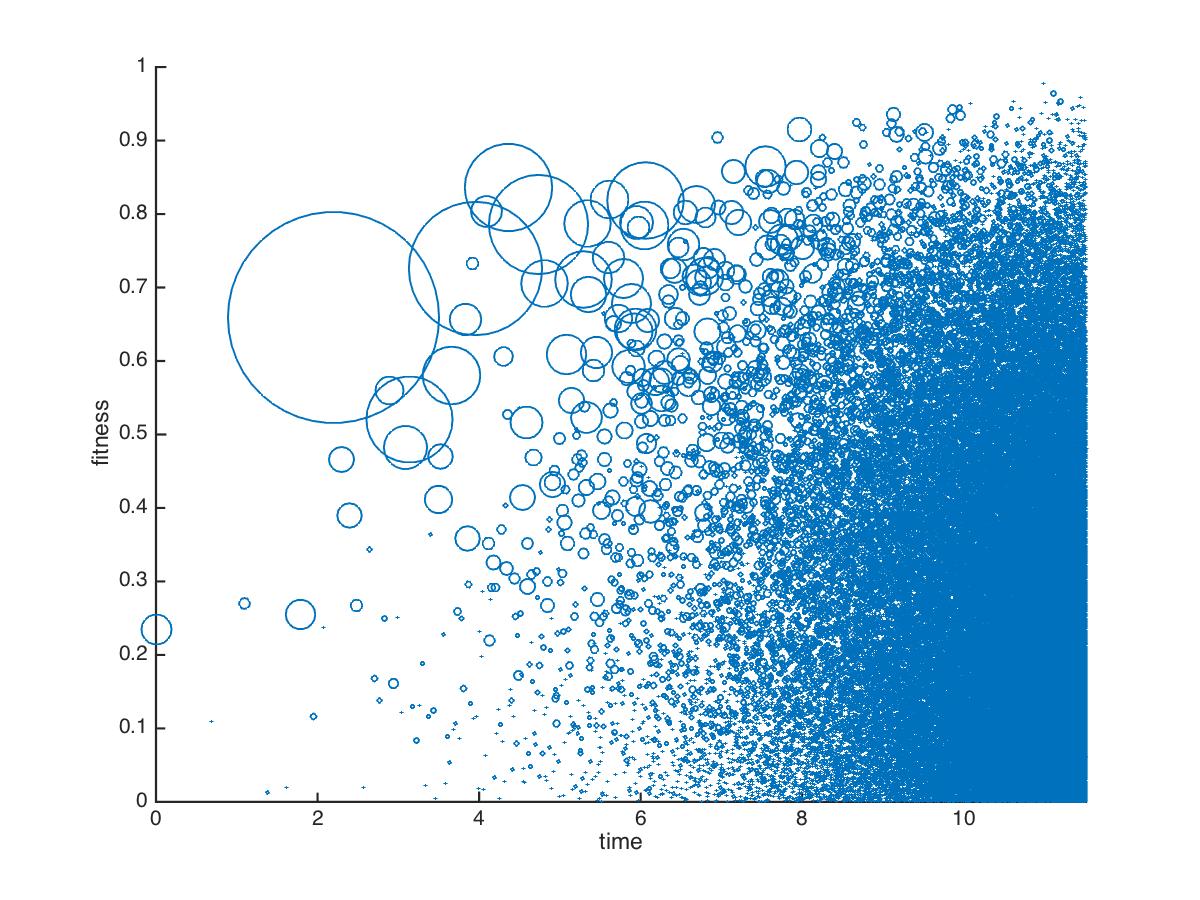}
\end{center}
\caption{A simulation of a reinforced branching process in the condensation case.  
Parameters are $\mu(dx)=3(1-x)^2\, dx$ and $\beta=\gamma=1$.
Each family is represented by a circle with area proportional to\\ its size at time $t=12$ and centred at its time of birth (horizontal axis) and its fitness (vertical axis).\\  Simulation courtesy of Anna Senkevich.} 
\label{fig:bubbles}
\end{figure}

We start with a heuristic consideration. Suppose $t>0$ is given. At any time $s\in(0,t)$ there are $\exp((\lambda^*+o(1))s)$ families in the system and by an extreme value calculation the largest fitness in this number of families  is of order $1-\exp(-(\frac{\lambda^*}{\alpha}+o(1))s)$. Until time~$t>s$ the family achieving this fitness has time $t-s$ to grow and therefore has size of order
\smash{$\exp(\gamma (t-s) (1-e^{-\frac{\lambda^*}{\alpha}s}))$}. We 
therefore expect the birth time $s$ of the maximal family at time $t$ to be around the maximiser of  this expression over $0<s<t.$ This maximum occurs roughly at time 
$s \sim \nicefrac\alpha{\lambda^\star} \log t$. 

For the rigorous results we replace this time~$s$ by the stopping time 
\[T(t):= \inf\bigl\{s\geq 0\colon M(s)\geq n(t)\bigr\} \quad \text{where}\quad  n(t):=\big\lceil\sfrac 1{\mu(1-t^{-1},1)}\big\rceil,\]
which 
allows us precise control over the number of families in the system. Note that
$T(t)\sim\nicefrac\alpha{\lambda^\star} \log t$, as can be seen  by putting 
$M(s)=\exp((\lambda^*+o(1))s)$ and $\log n(t)=(\alpha+o(1)) \log t$.
Our heuristics suggests that
the dominant families of the population at time $t$ are born in a window around time $T(t)$,
have fitness $F_n$ with $1-F_n$ of order~$1/t$, and  size of order  $\exp(\gamma (t-T(t)))$. 
\medskip
\pagebreak[3]

To confirm
this intuition we zoom into this window by considering the point process
\[\Gamma_t=\sum_{n=1}^{M(t)} \delta{\big(\tau_n-T(t), (t-\tau_n)(1-F_n),e^{-\gamma (t-T(t))}Z_n(t)\big)},\]
where $\delta(x)$ is the Dirac mass in~$x$.
\medskip


\begin{thm}[Poisson limit]\label{poisson}%
Under assumption (\ref{as})  the point process
$(\Gamma_t)_{t\geq 0}$ converges vaguely on the space $[-\infty,\infty]\times [0,\infty]\times (0,\infty]$ to the Poisson point process with intensity measure
\[d\zeta (s,f,z)= \alpha f^{\alpha-1} \lambda^{\star} e^{\lambda^{\star} s} e^{-z e^{\gamma (s+f)}}e^{\gamma (s+f)}  \, ds\,df\, dz.\]
\end{thm}
\medskip


\begin{remark}
{\rm Note the compactifications at $\pm\infty$ in Theorem~\ref{poisson}. As the limiting point process has a continuous density,
Theorem~\ref{poisson} implies that all mass of $\Gamma_t$ that asymptotically accumulates at infinity in one of the first two components,
must escape at zero in the last component, meaning that the only way points can disappear in the limit is because the corresponding family
size is small relative to the normalisation. }
\end{remark}
\bigskip

\begin{remark}
{\rm As there is no scaling in the first component of $\Gamma_t$, the limit theorem focuses on a time window of constant width around $T(t)$.
The theorem shows that this is wide enough to capture the largest family at time $t$. However, it turns out that in the condensation phase this is \emph{not} wide enough
to capture \emph{all} families that contribute to the condensate. This is why important questions on the emergence of the condensate remain open in this paper,
see for example the first two problems in Section~8.}
\end{remark}
\bigskip

Our Poisson limit result, Theorem~\ref{poisson}, readily  implies the following distributional limits
(denoted by $\Longrightarrow$) for the size, fitness and birth time of the largest family. 
\bigskip

\begin{cor}[Limits of family characteristics]\label{cor:as_laws}
\ \\[-2mm]
\begin{enumerate}[(i)]
\item Asymptotically, as $t\to\infty$,
\[
e^{-\gamma (t-T(t))} \max_{n\in\b N} Z_n(t) \Longrightarrow W^{-\frac{\gamma }{\lambda^{\star}}},
\]
where $W$ is exponentially distributed with parameter $\Gamma(\alpha+1)\Gamma(1+\frac{\lambda^{\star}}{\gamma } )(\lambda^{\star})^{-\alpha}$.
\item Under \eqref{cond}, denoting by $V(t)$ the fitness of the family of maximal size at time $t$, as $t\to\infty$, we have
\[t(1-V(t))\Longrightarrow V, \]
where $V$ is Gamma-distributed with scale parameter~$\lambda^\star$ and shape parameter $\alpha$.
\item  Denoting by $S(t)$ the birth time of the family of maximal size at time $t$, as $t\to\infty$, we have
\[ S(t)-T(t) \Longrightarrow U,  \]
where $U$ is a real valued random variable.
\end{enumerate}
\end{cor}
\bigskip

\begin{remark}
{\rm The birth time of the family of maximal size at time $t$ is of asymptotic order $T(t)+O(1)$ and hence 
(as seen above) of  leading order $\nicefrac{\alpha}{\lambda^\star} \log t$.  This answers the question of Borgs et al.~\cite{BC07} about  the rate at which 
new nodes with higher fitness become the leading influence in the population, see Figure~1 for a simulation.}
\end{remark}
\bigskip

\begin{thm}[The winner does not take it all]\label{nowinner}%
Under assumption (\ref{as}) the size of the largest family is negligible relative to the overall population size, i.e.\
$$
\lim_{t\to\infty} \frac{\max _{n\in\{1,\dots, M(t)\}} Z_n(t)}{N(t)} =0, \text{ in probability.}
$$
\end{thm}
\medskip

\begin{remark}
{\rm Theorem~\ref{nowinner} means that asymptotically no single family contributes a positive proportion of the 
total mass, hence if there is condensation it is always \emph{non-extensive}. This means in the context of Example~2 
that no vertex attracts a positive fraction of the edges in the network. This is at odds with the informal description of condensation in the preferential attachment networks  by Bianconi and Barab\'asi~\cite{BB01}, who are stating that \emph{`the fittest node [is] acquiring a finite fraction of the links,  independent of the size of the network.'}  It is also at odds with more recent work of Godr\`eche and Luck~\cite{GL10} who use
a numerical study and further analysis based on it to conclude that asymptotically there is an unbounded number of macroscopic families. Apparently the phenomenon we investigate here is too subtle to be reliably captured by non-rigorous techniques. In the context of Example~3 our theorem states that the proportion of balls of any colour goes to zero, uniformly over all colours.}
\end{remark}
\bigskip

The remainder of this paper is organised as follows. In Section~3 we prove Theorem~\ref{th:cond} by applying the theory of general branching processes. Section~4 contains an explicit construction of our model and uses this to give crude bounds on the rate of growth of the branching process. These will be used in Section~5 to derive a local version of Theorem~\ref{poisson}, i.e.\ a version without the essential compactifications of the underlying space. Section~6 provides the estimates need to compactify the space, and in Section~7 we complete the proof of  Theorem~\ref{poisson} and derive Corollary~\ref{cor:as_laws} and Theorem~\ref{nowinner}. The final section, Section~8, lists some interesting open problems.
\pagebreak[3]

\section{Proof of the condensation phenomenon}

In the last years a couple of techniques were developed to prove limit theorems for empirical distributions in networks and related structures, see for example Borgs et al.~\cite{BC07}, Bhamidi~\cite{Bh07} and Dereich and Ortgiese~\cite{DO14}. We now indicate how the theory of general branching processes can be used to prove Theorem~\ref{th:cond}. Our method is similar to the one described in~\cite{Bh07} but circumvents the use of multitype branching.
\smallskip

\pagebreak[3]

\subsection{The standard construction of the model}
We start with a construction of our model on an explicit probability space. Let 
\begin{itemize}
\item $F$ be a $\mu$-distributed random variable,
\item given $F$ let $Y=(Y(t) \colon t\geq 0)$ be an independent  Yule process with rate $\gamma F$, 
\item given $F, Y$ we define a simple point process $\Pi=(\Pi(t) \colon t\geq 0)$ as $\Pi=\Pi^{\ssup 1}+\Pi^{\ssup 2}$
where $\Pi^{\ssup 1}$ only jumps at the jumps of $Y$, and does so independently for every jump with probability
$\nicefrac{\beta+\gamma-1}{\gamma}$, and $\Pi^{\ssup 2}$ is an independent, inhomogeneous Poisson process with intensity measure
$(1-\gamma) F  Y(t) \, dt$.
\end{itemize}
\pagebreak[3]

We let $(\Omega,\cF,\b P)$ be the countable product of the joint law of $(F,Y,\Pi)$ 
and denote the coordinate process by $(F_n, Y_n, \Pi_n)$, for $n\in\mathbb N$.
We let $\tau_1=0$ and $Z_1(t)=Y_{1}(t)$ and iteratively define, for $n\in \{2,3,\dots\}$,
\begin{equation}\label{eq:tau_n}
\tau_{n}=\inf \{t>\tau_{n-1}: \exists m\in\{1,\dots,n-1\}\text{ with } \Delta \Pi_m(t-\tau_m)=1\}
\end{equation}
and
\[
Z_{n}(t)=
\begin{cases} Y_{n}(t-\tau_n), &\text{ if }t\geq \tau_n\\
0, &\text{ otherwise.}
\end{cases}
\]
We let $M(t)=\max\{ n \colon \tau_n \leq t\}$, set
$N(t)= \sum_{n=1}^{M(t)} Z_{n}(t)$, and denote by $\sigma_1, \sigma_2,\ldots$ the jump times of
$(N(t) \colon t\geq 0)$. It is obvious that this construction defines the reinforced branching process described 
in the introduction. Indeed $(Y_n(t - \tau_n) \colon t\geq \tau_n)$ gives the times of birth of new individuals in the $n$th family, 
and $(\Pi_n(t - \tau_n)\colon t\geq \tau_n)$ the times of creation of the new families which descend directly from the $n$th family.  
\smallskip

For later reference we now recall some facts about Yule processes.\smallskip

\begin{lem}\label{lem:doob}
Let $Y$ be a Yule process with rate $\lambda$. Then,
\begin{enumerate}[(a)]
\item $(e^{-\lambda t}Y(t))_{t\geq 0}$ is a uniformly integrable martingale.
\item The almost sure limit $\lim_{t\to\infty} e^{-\lambda t }Y(t)$ is standard exponentially distributed.
\item For $u\in[0,1)$ one has  
\begin{align}\label{eq1108-1}
\b E\Bigl[\sup_{t\geq 0} \exp\{ u e^{- \lambda t} Y(t)\}\Bigr]\leq \frac{4}{1-u}.
\end{align}
\item Denote by  $T_n = \inf\{s\geq 0 \colon Y(s)\geq n\}$. Then for every $\eps>0$  with high probability  as $\kappa\to\infty$ for all  $n_0,n\geq \kappa$ 
$$
\frac 1\lambda \log \frac n{n_0}-\eps\leq T_n-T_{n_0}\leq \frac 1\lambda \log  \frac n{n_0}+\eps.
$$
\end{enumerate}
\end{lem}

\begin{proof}
$(a)$ and $(b)$ are standard and proofs can be found in Athreya and Ney~\cite{AN}. Denote
the martingale limit in~$(b)$ by $A$. For the proof of $(c)$ note that  $(\exp\{ u e^{-t } Y(t)/2\} \colon t\geq0)$ is a sub-martingale by  Jensen's inequality. Doob's martingale inequality then gives
$$
\b E\Bigl[\sup_{t\geq 0} \bigl(\exp\{ u e^{-t } Y(t)/2\}\bigr)^2\Bigr]\leq 
4  \, \b E\bigl[\exp\{u A\}\bigr] =\frac{4}{1-u}.
$$
To prove $(d)$ we may assume, without loss of generality, that $\lambda=1$.
Consider the martingale given by
$\xi_t=e^{-t} Y(t)$, and let $R(\kappa):=\sup\{ \frac {\xi_s}{\xi_u} :  s, u\geq T_\kappa\}$. By $(b)$, $(\xi_t)_{t\geq 0}$ has an almost surely finite, strictly positive limit and one has  $\lim_{\kappa\to\infty} R({\kappa})=1$, in probability. Further
$$
\frac {Y(t+T_{n_0})}{Y(T_{n_0})}= e^{t} \, \frac {\xi_{t+T_{n_0}}}{\xi_{T_{n_0}}}\in\Bigl[\frac1 {R(\kappa)} \,e^t, R(\kappa)\,e^t\Bigr].
$$
An application of  the estimate for all $n,n_0\geq \kappa$ with  $t=T_n-T_{n_0}$  gives that
$$
\frac {n}{n_0}\in \Bigl[\frac1 {R(\kappa)} \,e^{T_n-T_{n_0}}, R(\kappa)\,e^{T_n-T_{n_0}}\Bigr].
$$
Taking logarithms and recalling that $R(\kappa)$ tends to $1$ yields the statement.
\end{proof}

\subsection{General branching process theory}

The processes $(M(t) \colon t>0)$ is a general branching process, or Crump-Mode-Jagers process, 
with the laws of offspring times given by the point process $(\Pi(t) \colon t>0)$. Nerman~\cite{Nerman}
provides a strong law of large numbers for this class of processes under the assumption that
there exists  ${\lambda^*}>\gamma$, called the \emph{Malthusian parameter}, such that
$$\int_0^\infty e^{-{\lambda^*} s}  \, {\mathbb E} \Pi(ds)=1.$$
An easy calculation 
(which we skip since it is already carried out in detail in the particular case of Example~1 above)
shows that this is equivalent to
$$\frac{\beta}{ \beta+\gamma} \int_0^1 \frac{\lambda^{\star}}{\lambda^{\star}-\gamma x} \, d\mu(x) =1.$$ 
Suppose that $\phi=\phi[F, Y, \Pi]\colon [0,\infty) \to \N_0$ is a cadlag process taking values in the nonnegative integers, such that $\phi(t)$ is interpreted as a score assigned to a family $t$ time units 
after its foundation. We assume the function $t\mapsto \E[\phi(t)]$ is almost everywhere continuous and 
there exists $h\colon[0,\infty)\to(0,\infty)$ integrable, bounded and non-increasing such that
$$\E\Big[ \sup_{t\geq 0} \frac{e^{-{\lambda^\star} t} \phi(t)}{h(t)}\Big]<\infty.$$

Letting $\phi_n=\phi[F_n, Y_n, \Pi_n]$ we define the \emph{score of the population} 
at time $t$ as
$$Z^\phi(t)= \sum_{n \colon \tau_n<t} \phi_n(t-\tau_n).$$
We define
$$m^\phi_\infty = \frac{\int_0^\infty e^{-{\lambda^{\!\star}} t} \E \phi(t)\, dt}{\int_0^\infty te^{-{\lambda^{\!\star}} t}\, \E\Pi(dt)}.$$
Nerman~\cite{Nerman} shows that there exists a positive random variable $W$, not depending on $\phi$,  such that
\begin{equation}\label{nerman}
\lim_{t\uparrow\infty} e^{-{\lambda^{\!\star}} t} Z^\phi(t) = W \, m^\phi_\infty \qquad \mbox{ almost surely.}
\end{equation}
Under the assumption that \eqref{cond} fails and $\lambda^\star>\gamma$ we can apply this result with
$\phi(t)=Y(t)$, which satisfies the assumptions by Lemma~\ref{lem:doob}, to get 
$$\lim_{t\uparrow\infty} e^{-{\lambda^{\!\star}} t} N(t) = W \, m^\phi_\infty \qquad \mbox{ almost surely.}$$
Choosing $\phi(t)=Y(t) \1\{F\geq 1-x\}$, for $0<x<1$, and combining with the above gives
$$\lim_{t\uparrow\infty}\Xi_t[1-x,1] =\frac{\beta}{\beta+\gamma} \int_{1-x}^1\frac{\lambda^\star}{\lambda^{\star}-\gamma x} \, d\mu(x) \qquad \mbox{ almost surely,}$$
as required.

\subsection{A coupling technique}

To extend results to the case when $\lambda^\star=\gamma$, or when no Malthusian parameter is available, we use a coupling technique.  We look at the reinforced branching process 
at the times $(\sigma_n)_{n\geq 1}$ of the birth events and abbreviate \smash{$\widehat\Xi_n:=\Xi_{\sigma_n}$. }
\smallskip

Fix $\varepsilon>0$.  We define a discrete-time branching process 
whose empirical fitness distribution $\widehat\Xi^{\ssup{\varepsilon}}_n$ has the property
that, for all $n\geq 0$, \smash{$(\widehat\Xi_n, \widehat\Xi_n^{\ssup{\varepsilon}})\in\mathcal S$}, 
where $\mathcal S$ is the subset of the set of pairs of probability measures on $[0,1]$ defined by
$\mathcal S := \{(\nu, \mu) \colon 
\nu([a,b])\geq \mu([a,b])$ for all $a, b \in[0,1-\varepsilon)\}.$
Let $(U_n)_{n\geq 1}$ be a sequence of i.i.d.\ random variables uniformly distributed on $[0,1]$. 
\smallskip

At time zero, the new process contains one family of fitness $F_1\1\{F_1< 1-\varepsilon\} + 
\1\{F_1 \geq 1-\varepsilon\}$ and thus 
\smash{$(\widehat\Xi_0,\widehat\Xi^{\ssup{\varepsilon}}_0)\in\mathcal S$}.
Assume now that, \smash{$(\widehat\Xi_n, \widehat\Xi^{\ssup{\varepsilon}}_n)\in\mathcal S$.}
We construct the new process at time $n+1$ as follows:
\begin{itemize}
\item if a new family of fitness $f$ is born at time $n+1$ (in the original process), 
then we add in the (new) process a new family of fitness 
$f \1\{f< 1-\varepsilon\} + \1\{ f \geq1-\varepsilon\}$ born at time $n+1$;
\item if an individual of fitness larger than $1-\varepsilon$ is born at time $n+1$ in the original process, 
then we add a new individual of fitness 1 born at time $n+1$;
\item if an individual of fitness $f< 1-\varepsilon$ is born at time $n+1$ in the original process, then
if
\[U_{n+1}\leq \left(\frac{\widehat\Xi_n^{\ssup \eps}(\{f\})}{\int_0^1 x\,{\rm d}\widehat\Xi^{\ssup \eps}_n(x)}\right) \left(\frac{\widehat\Xi_n(\{f\})}{\int_0^1 x\,{\rm d}\widehat\Xi_n(x)}\right)^{-1},\]
we add an individual of fitness $f$ born at time $n+1$, otherwise, we add an individual of fitness $1$.
\end{itemize}
By construction, $(\widehat\Xi_{n+1}, \widehat\Xi^{\ssup{\varepsilon}}_{n+1})\in\mathcal S$.
It is now easy to check that the new process is the discrete-time version of the reinforced branching process
with fitness distribution~$\mu_{\varepsilon}:=\1[0,1-\varepsilon)\mu + \mu[1-\varepsilon, 1]\delta_1$.
Since
\[\frac{\beta}{\beta+\gamma}\int_0^1 \frac{d\mu_{\varepsilon}(x)}{1-x} = \infty,\]
the new process admits a Malthusian parameter $\lambda_{\varepsilon}$ 
and $\lambda_{\varepsilon} \downarrow \gamma$ as $\varepsilon\downarrow 0$.
We deduce that, for all $0\leq a, b <1-\eps$,  we have
\[\lim_{n\to\infty} \widehat\Xi^{\ssup{\varepsilon}}_n\big([a,b]\big) =
\lim_{t\to\infty} \widehat\Xi^{\ssup{\varepsilon}}_t\big([a,b]\big) 
= \frac{\beta}{\beta+\gamma} \int_a^b \frac{\lambda_\varepsilon }{\lambda_\varepsilon -\gamma x}\, d\mu(x)\]
almost surely.
For all $0\leq a, b <1$ and $0<\varepsilon < 1-b$, we thus have
\[\liminf_{t\to\infty} \Xi_t\big([a,b]\big) 
=\liminf_{n\to\infty} \widehat\Xi_n\big([a,b]\big) 
\geq \lim_{n\to\infty} \widehat\Xi_n^{\ssup{\varepsilon}}\big([a,b]\big)
= \frac{\beta}{\beta+\gamma} \int_a^b \frac{\lambda_\varepsilon}{\lambda_\varepsilon -\gamma x}\, d\mu(x).\]
Letting $\varepsilon\downarrow 0$ gives the lower bound. A similar argument gives a coupling with
the reinforced branching process with  $\mu$ replaced by
$\mu^{\ssup\varepsilon} = 
\1[0,1-\varepsilon) \mu +  \mu[1-\varepsilon,1]\delta_{1-\varepsilon},$
and provides an upper bound, which is enough to conclude the proof of Theorem~\ref{th:cond}.

\section{Estimates for the number of families in the population}\label{sec:approx_tau_n}

The main difficulty in our model is that the time of birth of the $n$th family is not known with good accuracy. 
We now give a rough deterministic bound for the births occurring around the stopping times~$T(t)$. 
\bigskip

\begin{prop}\label{lem:approx_tau_n}
For all $\varepsilon\in(0,\lambda^\star)$, we have with high probability as $n_0\to\infty$, for all $n\geq n_0$,
\[\frac{1}{\lambda^{\star}+\eps} \log \frac{n}{n_0}-\eps 
\leq \tau_n-\tau_{n_0} \leq \frac1{\lambda^{\star}-\varepsilon} \log \frac{n}{n_0}+\eps,\]
and, for all $1 \leq n\leq n_0$,
\[\frac{1}{\lambda^{\star}-\eps} \log \frac{n}{n_0} -\eps
\leq \tau_n-\tau_{n_0} \leq \frac1{\lambda^{\star}+\varepsilon} \log \frac{n}{n_0}+\eps.\]
\end{prop}

\medskip

\begin{proof}
Recall that, when $t$ goes to infinity, $M(t)= e^{(\lambda^\star+o(1))t}$, implying that
$$
\tau_{n_0}=\frac {1+o(1)}{\lambda^\star} \log n_0
$$
which implies the estimate for arbitrarily fixed $n$ as $n_0$ goes to infinity.
It thus suffices to prove the statement for all $n, n_0\geq \kappa$ with high probability as $\kappa$ goes to infinity.
Fix  $\eps\in(0,\lambda^\star)$ and $\kappa\in\N$ and set $\lambda_\pm^\star:=\lambda^\star\pm\eps$.
The stochastic process $(M(s))_{s\geq 0}$ is a pure birth process with continuous compensator 
$$\gamma(s)=\int_0^s \beta N(u)  \int _0^1 x\, \Xi_{u}(dx)\,du.$$
 We consider the  stopping time  $S$ given by
$$
S=\inf\Bigl\{s\geq \tau_\kappa\colon  \frac{\beta N(s)}{ M(s)}\int _0^1 x\, \Xi_s(dx) \not \in [ \lambda ^\star_-, \lambda^\star_+]\Bigr\},
$$
and note that, by Theorem~\ref{th:cond}, $S$ is infinite with high probability as $\kappa\to\infty$.

Let $w_n:=\tau_{n+1}-\tau_n$ be the  inter-arrival times of $(M(s))_{s\geq 0}$. Observe that 
\begin{align*} 
\gamma(\tau_{n+1})-\gamma(\tau_n)= \int_{\tau_n}^{\tau_{n+1}} \beta N(s)  \int _0^1 x\, \Xi_s(dx)\, ds \in [\lambda^\star_-(\tau_{n+1}-\tau_n)n,  \lambda^\star_+ (\tau_{n+1}- \tau_{n})n]
\end{align*}
provided that $\tau_{n+1}\leq S$. Defining $$w_n^\pm:=\frac {\gamma(\tau_{n+1})-\gamma(\tau_n) }{\lambda^\star_\pm n},$$ we infer that $w^+_n\leq w_n\leq w^-_n$ for all $n$ such that 
$\tau_{n+1}\leq S$,  and the sequences $(w^\pm_n)_{n\in\N}$ consist of independent exponentials with respective parameter $\lambda^\star_\pm n$ which are the inter-arrival times of Yule processes $(Y_\pm(s))$ of respective rate $\lambda_\pm^\star$. By Lemma~\ref{lem:doob}, with high probability  as $\kappa\to\infty$, for all~$n,n_0\geq \kappa$,
$$
 \frac 1{\lambda^\star_\pm} \log \frac n{n_0}-\eps \leq T^\pm_n-T^\pm_{n_0}\leq \frac 1{\lambda^\star_\pm} \log \frac n{n_0} +\eps,
$$
where $(T^\pm_n)_{n\in\N}$ denotes the ordered sequence of jump times of $(Y_\pm(s))_{s\geq 0}$.
Hence, for all $n\geq n_0\geq \kappa$,
$$
\tau_n-\tau_{n_0}=\sum_{k=n_0}^{n-1} w_k \leq \sum_{k=n_0}^{n-1} w_k^- = T_n^-- T_{n_0}^- \leq \frac 1{\lambda^\star_-} \log \frac n{n_0}+\eps.
$$
Similarly, it follows that with high probability, for $\kappa\leq n\leq n_0$,
$$
\tau_n-\tau_{n_0}\leq T_n^+- T_{n_0}^+ \leq \frac 1{\lambda^\star_+} \log \frac n{n_0}+\eps.
$$
The converse bound follows in complete analogy.
\end{proof}

\section{Local convergence}

The aim of this section is to prove the following proposition.
\medskip

\begin{prop}\label{prop:Poisson_local_Gamma}
Under assumption (\ref{as}) one has convergence in distribution of the point processes
\[\Gamma_t=\sum_{n=1}^{M(t)} \delta_{(\tau_n-T(t),(t-\tau_n)(1-F_n), e^{-\gamma (t-T(t))}Z_n(t))}\]
to the Poisson point process with intensity
\[d \zeta(s,f,z)= \alpha f^{\alpha-1}  \lambda^{\star} e^{\lambda^{\star} s}  e^{-ze^{\gamma (f+s)}} e^{\gamma (f+s)}\, ds\, df\, dz,\]
where we endow the set of locally finite  measures on $(-\infty,\infty)\times [0,\infty)\times [0,\infty]$ with the topology of vague convergence.
\end{prop}

\vspace{.5\baselineskip}
Proposition \ref{prop:Poisson_local_Gamma} is a straightforward consequence of the following result.
\medskip

\begin{prop}\label{prop:Poisson_local}
Under assumption (\ref{as}) we have vague convergence of the point process
\[\Psi_t=\sum_{n=1}^{M(t)} \delta_{(\tau_n-T(t),(t-\tau_n)(1-F_n), e^{-\gamma  F_n(t-\tau_n)}Z_n(t))}\]
to the Poisson point process with intensity
\[d \zeta^{\star} (s,f,z)= \alpha f^{\alpha-1}  \lambda^{\star} e^{\lambda^{\star} s}  e^{-z}\, ds\, df\, dz\]
on $(-\infty,\infty)\times [0,\infty)\times [0,\infty]$.
\end{prop}

\begin{proof}[Proof of Proposition \ref{prop:Poisson_local_Gamma}]
This follows directly from the fact that the point process $\Gamma_t$ is the image of $\Psi_t$ by the continuous function
$(s, f, z) \mapsto (s, f, e^{-\gamma (s+f)}z),$
and that $\zeta$ is the image of $\zeta^{\star}$ by the same continuous function.
\end{proof}

The proof of Proposition~\ref{prop:Poisson_local} consists of the following two steps. 
We approximate $\Psi_t$ by a point process 
\[\Psi^{\star}_t 
= \sum_{n \in \b N}  \,\delta_{(\tau_n^{\star}(t)-T(t), t(1-F_n), \xi_n)},\]
where the birth times $\tau_n$ are replaced by approximate birth times
\[\tau_n^{\star}(t) = T(t) + \frac{\log \nicefrac{n}{n(t)}}{\lambda^{\star}}\]
and the rescaled family sizes $ e^{-\gamma  F_n (t-\tau_n)}Z_n(t)$ by their limits
$$\xi_n=\lim_{u\to\infty} e^{-\gamma  F_n (u-\tau_n)}Z_n(u).$$
In the approximating process $\Psi^{\star}_t$
the components are decoupled, which makes it easier to study.
We prove that
\begin{itemize}
\item[(1)] this approximation process converges vaguely to the Poisson point process of intensity $\zeta^{\star}$, and
\item[(2)] $\Psi^{\star}_t$ is close enough to $\Psi_t$ to imply Proposition~\ref{prop:Poisson_local}.
\end{itemize}
The two steps correspond to the two lemmas below.
\bigskip

%

\begin{lem}\label{lem:cv_psi_star} $(\Psi_t^{\star})_{t\geq 0}$ converges vaguely on $[-\infty,\infty)\times [0,\infty)\times [0,\infty]$ to the Poisson point process with intensity~$\zeta^{\star}$. 
\end{lem}

\begin{proof}
We apply Kallenberg's theorem, see~\cite[Proposition~3.22]{Resnick}. Since $\zeta^\star$ is diffuse,
to prove Lemma~\ref{lem:cv_psi_star}, it is enough to show that, for every
precompact relatively open box $B\subset [-\infty,\infty)\times [0,\infty)\times [0,\infty]$,
\begin{enumerate}[(a)]
\item $\mathbb P(\Psi^{\star}_t(B)=0) \to 
\exp(-\zeta^{\star}(B))$, as $t\uparrow\infty$, and
\item $\mathbb E \Psi^{\star}_t(B) \to \zeta^{\star}(B)$, as $t\uparrow\infty$.
\end{enumerate}
It suffices to consider nonempty  boxes $B$ of the form $(a_0,a_1)\times (b_0,b_1)\times (c_0,c_1)$ since, almost surely, neither  the point processes $\Gamma_t$ nor the limiting Poisson process put points on the boundary $\partial ([-\infty,\infty)\times [0,\infty)\times [0,\infty])$. 
Here $a_0=-\infty$ and $c_1=\infty$ is an allowed choice.
Note that 
\[\zeta^{\star}(B)= 
(e^{\lambda^{\star}a_1}-e^{\lambda^{\star}a_0})
(b_1^{\alpha}-b_0^{\alpha})
(e^{-c_0}-e^{-c_1}).\]
(a) By the construction of the probability space at  the beginning of Section~\ref{sec:approx_tau_n},  $(F_n, \xi_n)_{n\geq 1}$ is a sequence of i.i.d. random variables with each $F_n$ being independent of $\xi_n$. Hence,
\begin{align*}\mathbb P(\Psi^{\star}_t(B) = 0)
&= \prod_{n(t)e^{\lambda^{\star}a_0}< n < n(t)e^{\lambda^{\star}a_1}}
\mathbb P\big(t(1-F_n)\notin (b_0,b_1) \text{ or }\xi_n \notin (c_0, c_1)\big)\\
&= \bigl(1- \P\big(t(1-F_1)\in (b_0,b_1)\big) \,\P\big(\xi_1 \in (c_0, c_1)\big)\bigr)^{r_{a_0,a_1}(t)},
\end{align*}
where $r_{a_0,a_1}(t)$ denotes the number of elements $n\in\N$ with $n(t)e^{\lambda^{\star}a_0}< n < n(t)e^{\lambda^{\star}a_1}$. \pagebreak[3]

We note that, as $t\to\infty$, we have $r_{a_0,a_1}(t)\sim  (e^{\lambda^\star a_1}-e^{\lambda^\star a_0})/\mu(1-t^{-1},1)$ and,  in view of Assumption~\ref{as}, 
\begin{equation}\label{eq:RV_frac}
\frac{\mu((1-\nicefrac{b_1}{t}, 1-\nicefrac{b_0}{t}))}{\mu(1-\nicefrac1t, 1)} \sim b_1^{\alpha}-b_0^{\alpha} \quad \text{ as }t\uparrow\infty.
\end{equation}
Further $\P(\xi_1\in (c_0,c_1))=e^{-c_0}-e^{-c_1}$.
Thus, as $t\to\infty$,
\begin{align*}
\mathbb P(\Psi^{\star}_t(B) = 0)
&= \left(1-\mu(1-\nicefrac{b_1}{t}, 1-\nicefrac{b_0}{t})\left(e^{-c_0}-e^{c_1}\right)
\right)
^{r_{a_0,a_1}(t)}\\
&\sim \exp\left(- (e^{\lambda^{\star}a_1}-e^{\lambda^{\star}a_0}) \ \frac{\mu((1-\nicefrac{b_1}{t}, 1-\nicefrac{b_0}{t}))}{\mu(1-\nicefrac1t, 1)}\ \left(e^{-c_0}-e^{c_1}\right)
\right)\\
&\to \exp(-\zeta^{\star}(B)).
\end{align*}

(b) To compute the limit of $\b E[\Psi^{\star}_t(B)]$ we apply the asymptotic estimates from above,
\begin{align*}
\b E[\Psi^{\star}_t(B)]&= 
\sum_{n(t) e^{\lambda^{\star}a_0}< n< n(t)e^{\lambda^{\star}a_1}} 
\mu([1-\nicefrac{b_1}{t},1-\nicefrac{b_0}{t}]) \,\b P(\xi_n \in[c_1,c_2])\\
&= r_{a_0,a_1}(t)\, \mu([1-\nicefrac{b_1}{t},1-\nicefrac{b_0}{t}]) \,\b P(\xi_1 \in[c_1,c_2])\to \zeta^\star(B).
\end{align*}
\end{proof}

\begin{lem}\label{lem:approx_psi}
For all Lipschitz continuous, compactly supported functions $f:(-\infty,\infty)\times [0,\infty)\times [0,\infty] \to \mathbb R$,
\[\left|\int f\, d\Psi_t^{\star} - \int f\,d \Psi_t\right| \to 0 \text{ in probability, as }t\uparrow\infty.\]
\end{lem}

\begin{proof}
Let $f$ be a Lipschitz continuous function supported on  
$K= [-a,a]\times [0,b]\times [0,\infty]$ for $a,b\geq 1$.
We have
\begin{align}\begin{split}\label{eq0809-1}
&\Bigl|\int f\, d\Psi_t -\int f\, d\Psi_t^{\star}\Bigr|\\
&\hspace{1cm}\leq \sum_{n=1}^{M(t)} \left|f\left(\tau_n - T(t), (t-\tau_n)(1-F_n), e^{-\gamma F_n(t-\tau_n)}Z_n(t)\right)
- f\Big(\tau_n^\star(t)-T(t), t(1-F_n),\xi_n\Big)\right|\\
&\hspace{1cm}\leq c \sum_{n\in I(t)} \bigl( | \tau_n - \tau_n^\star(t)|+  \tau_n (1-F_n) + |  e^{-\gamma F_n(t-\tau_n)}Z_n(t)-\xi_n|\bigr) ,
\end{split}\end{align}
where $c$ is the Lipschitz constant of the function $f$ and
 $I(t)$ is the random set of indices $n\in \mathbb N$ such that
\begin{equation}\tag{a}\label{eq:condition1}
|\tau_n - T(t)|\leq a \ \text{ and } \ (t-\tau_n)(1-F_n)\leq b
\end{equation}
or
\begin{equation}\tag{b}\label{eq:condition2}
|\tau_n^{\star}(t)-T(t)|\leq a \ \text{ and } \ t(1-F_n)\leq b. 
\end{equation}

For $\eps\in (0,1/2)$ we denote by $\c T_\eps(t)$ the event that the following properties hold,
\begin{itemize}
\item $|\tau_n-\tau_n^\star(t)|\leq \eps (1+| \tau^\star_n(t) -T(t)|)$ for all $n\in \N$ and
\item $T(t)\leq t/3$.
\end{itemize}
We note that, in view of Proposition~\ref{lem:approx_tau_n}, $\c T_\eps(t)$ holds with high probability as $t\to\infty$. 

Let now
$\bar I(t):=\{n\in\N\colon |\tau^\star_n(t)-T(t)|\leq 2a, t(1-F_n)\leq 2b\}.$
We show that for $t\geq 6a$ one has
$I(t)\subset \bar I(t)$ on $\c T_\eps(t)$.
Suppose that $n\in I(t)$ and that $\c T_\eps(t)$ holds. 
It suffices to consider the case where condition~(\ref{eq:condition1}) is satisfied. Then
$b\geq (t-\tau_n)(1-F_n)\geq (t-T(t)-a) (1-F_n)\geq t(1-F_n)/2$, which proves that the second inequality in the definition of $\bar I(t)$ is satisfied.
Let us further assume that $n\leq n(t)$. Then
$a\geq |\tau_n-T(t)|\geq (1-\eps) |\tau_n^\star(t)-T(t)|-\eps,$
which implies that 
\[|\tau_n-\tau_n^{\star}(t)|\leq \varepsilon(1+|\tau_n^\star(t) -T(t)|)\leq \varepsilon
\big(1+\sfrac{a+\varepsilon}{1-\varepsilon}\big)\leq 4a\varepsilon.\]
The same inequality holds if $n\geq n(t)$ and thus we have proved that $I(t)\subset \bar I(t)$ on $\c T_\eps(t)$ for all $t\geq 6a$.\pagebreak[3]

We now consider the sum on the right hand side of \eqref{eq0809-1}, but taken over all $n\in\bar I(t)$. 
First note that, for $n\in \bar I(t)$, on~$\c T_\eps(t)$, we have
$\tau_n\leq T(t)+(1+\eps) (\tau_n^\star(t)-T(t)) + \eps\leq \nicefrac t3 + 3a, \text{ if }n\geq n(t),$
and $\tau_n(1-F_n)\leq   \nicefrac{2b}t (T(t) +3a).$
Second we let, for $n\in\N$ and $s\geq 0$,
$$
R_n(s):= \sup_{u\geq s} \big|e^{-u} Z_n(\tau_n + u/F_n)-\xi_n \big| 
$$ 
and using that for $t\geq 4b$ and $n\in\bar I(t)$ one has $F_n\geq \nicefrac12$ we conclude that for all $t\geq \max(4b, 18a)$,
$$
|  e^{-\gamma F_n(t-\tau_n)}Z_n(t)-\xi_n|\leq R_n(\gamma  F_n(t-\tau_n)) \leq R_n(\gamma t/4).
$$
Hence we get that, for sufficiently large $t$, on $\c T_\eps(t)$,
\begin{align*}
&\Bigl|\int f\, d\Psi_t -\int f\, d\Psi_t^{\star}\Bigr|\\
&\hspace{1cm}\leq c \sum_{n\in \bar I(t)} \bigl( | \tau_n - \tau_n^\star(t)| +  \tau_n (1-F_n) + |  e^{-\gamma F_n(t-\tau_n)}Z_n(t)-\xi_n|\bigr) \\
&\hspace{1cm}\leq c \sum_{n\in \bar I(t)} \Bigl(4a \eps +  \frac{2b}{t}(T(t)+3a) + R_n(\gamma  t/4)\Bigr) \\
&\hspace{1cm}\leq c \, |\bar I(t)| \Bigl(4a \eps +  \frac{2b}{t}(T(t)+3a)\Bigr)  +c \sum_{n\in \bar I(t)} R_n(\gamma  t/4).
\end{align*}
By construction, the random processes $(R_n)_{n\geq 1}$ are independent of $(F_n)_{n\geq 1}$ and thus also of the random set $\bar I(t)$. We recall that, by Proposition~\ref{prop:Poisson_local}, $|\bar I(t)|$ converges in distribution to a Poisson distribution and $\lim_{s\to\infty} R_n(s)=0$, almost surely. Hence, 
$$
\lim_{t\to\infty} \sum_{n\in \bar I(t)} R_n(\gamma s t/4)=0, \text{ in probability}.
$$
Since further $\lim_{t\to\infty} T(t)/t=0$, almost surely, we conclude that, with high probability,
$$
\Bigl|\int f\, d\Psi_t -\int f\, d\Psi_t^{\star}\Bigr|\leq 8\eps a c(|\bar I(t)|+1).
$$
Recalling again that $|\bar I(t)|$ converges in distribution and that $\eps\in(0,1/2)$ can be made arbitrarily small we obtain convergence to zero in probability, as $t\uparrow\infty$.
\end{proof}

\begin{proof}[Proof of Proposition~\ref{prop:Poisson_local}] 
Let $f\colon \b R\times [0,\infty)\times [0,\infty)\to\b R$ be Lipschitz continuous and compactly supported.
Combining Lemmas~\ref{lem:cv_psi_star} and~\ref{lem:approx_psi}, together with Slutzky's lemma we get the desired result,
\[\int f\, d\Psi_t \Rightarrow \int f\,d \mathtt{PPP}(\zeta^{\star})\quad\text{ as }t\uparrow\infty,\]
where $\mathtt{PPP}(\zeta^{\star})$ denotes the Poisson point process with intensity $\zeta^\star$.
\end{proof}


\section{Negligibility of families outside the main window}
To deduce Theorem~\ref{poisson} from Proposition~\ref{prop:Poisson_local_Gamma}, one has to control the contribution of the point process near the closed boundaries of $[-\infty, +\infty]\times [0,+\infty]\times (0,+\infty]$. We prove that the families that are born too late, or that are not fit enough, are too small to contribute in the limit. They \emph{get absorbed by the open lower bound of the third coordinate}.
We first provide a simple calculation, which is at the heart of our proofs.
Recall that
\[n(t)=\left\lceil  \frac1{\mu(1-\nicefrac1t,1)}\right\rceil.\]

\begin{lem}\label{lem:arg_reg_var}
Let $F$ be a random variable with law $\mu$.
There exists $t_0>0$ such that, for all  $C\geq0$, $D>0$, there exists $K=K(C,D)>0$ such that
$$\mathbb E \left[\1{\{F\leq 1-\sfrac C t\}} e^{-D(1-F)t}\right] \leq \frac{K}{n(t)}  \mbox{ for all } t \geq t_0.
$$
Moreover, for all $D$, we have $\lim_{C\uparrow \infty} K(C,D)=0$.
\end{lem}

To prove this lemma, we need Potter's bound, see Theorem~1.5.6$\,(ii)$ in~\cite{Bingham}. Since $\mu$ verifies Equation \eqref{as}, and is bounded from zero and infinity on every compact set of $(0,1]$, for all $\delta>0$, there exists a constant $\Lambda=\Lambda(\delta)$ such that, for all $0<x,y\leq 1$,
\begin{equation}\label{eq:Potter}
\frac{\mu(1-y, 1)}{\mu(1-x, 1)}\leq \Lambda \left(\frac{y}{x}\right)^{\alpha+\delta}.
\end{equation}

\begin{proof}[Proof of Lemma~\ref{lem:arg_reg_var}]
Fix $\delta>0$. We have
\begin{align*}
\mathbb E \left[\1{\{F\leq 1-\sfrac C t\}} e^{-D(1-F)t}\right]
&= \int_0^{\infty} \mathbb P\left(\1{\left\{F\leq 1-\sfrac C t\right\}} e^{-D(1-F)t}\geq x\right) dx
= \int_0^{\infty} \mathbb P\left(1-\frac{\log \nicefrac1x}{Dt}\leq F\leq 1-\frac C t \right) dx\\
&\leq \int_0^{e^{-CD}} \mu\left(1-\sfrac{\log \nicefrac1x}{Dt}, 1\right) dx
\leq \Lambda(\delta) 
\mu(1-\nicefrac1t, 1) \int_0^{e^{-CD}} \left(\sfrac1D \log \sfrac1x\right)^{\alpha+\delta} dx,
\end{align*}
where $\Lambda(\delta)$ is defined by the Potter's bound (see Equation~\eqref{eq:Potter}).
Changing variables $y= \sfrac1 D \log \sfrac 1x$ we get
\[
\mathbb E \left[\1{\{F\leq 1-\sfrac C t\}} e^{-D(1-F)t}\right]
\leq D \Lambda(\delta)
\mu(1-\nicefrac1t, 1) \int_C^{\infty} y^{\alpha+\delta} e^{-Dy} dy.
\]
Recall that $n(t) = \lfloor\nicefrac1{\mu(1-1/t, 1)}\rfloor$, and let
$K(C, D):=D \Lambda(\delta) \int_C^{\infty} y^{\alpha+\delta} e^{-Dy} dy$
to conclude the proof.
\end{proof}

\subsection{Contribution of the  unfit families}

\begin{lem}\label{lem:unfit}
For every $\eta>0$ and $c>0$ there exists $\kappa>0$ such that, for all 
sufficiently large $t$, we have
\[
\b P\Bigl(\max_{n\leq M(t)} \1{\{F_n\leq 1-\nicefrac\kappa t\}} \, Z_n(t)\geq c \,e^{\gamma (t-T(t))}\Bigr)
\leq \eta.
\]
\end{lem}

\begin{proof}
Let $c>0$ and $\kappa>0$. We analyse the event that there exists a family with fitness $F_n\leq 1-\nicefrac\kappa t$ and size
$e^{-\gamma (t-T(t))}  Z_n(t) \geq c$.
To this end, we define the time-shifted version $(Z^\star_n(t) \colon t \in \mathbb R)$ of the size of the $n^{\text{th}}$ family~by
\[
Z_n^{\star}(t):=Z_n(t+\tau_n-\tau_n^{\star}),\]
where
\[\tau_n^{\star}:= 
\begin{cases}
T(t) + \frac1{\lambda^\star - \varepsilon} \log \frac{n}{n(t)}-\varepsilon & \text{ if } n\leq n(t),\\
T(t) + \frac1{\lambda^\star + \varepsilon} \log \frac{n}{n(t)}-\varepsilon & \text{ if } n\geq n(t).
\end{cases}
\]
In view of Proposition~\ref{lem:approx_tau_n}, we have, with high probability as $t\to\infty$, that 
$Z_n^{\star}(t) \geq Z_n(t)$ for all $n\geq 0$.
Recalling the construction of the probability space at the beginning of Section~\ref{sec:approx_tau_n}, we note that the family $(A_n)_{n\geq 1}$ given by
\[A_n=\max_{s\geq \tau_n^{\star}}\frac{Z_n^{\star}(s)}{e^{\gamma F_n(s-\tau_n^{\star})}}
=\max_{s\geq \tau_n}\frac{Z_n(s)}{e^{\gamma F_n(s-\tau_n)}}\]
forms a sequence of i.i.d. random variables which is independent of $(F_n)_{n\geq 1}$. Further,
\begin{align*}
\bigl\{Z_n(t)\geq c\,e^{\gamma (t-T(t))}\bigr\} & \subset
\bigl\{Z_n^{\star}(t)\geq c\,e^{\gamma (t-T(t))}\bigr\} \\
& \subset \bigl\{A_n\geq c\, e^{\gamma [(t-T(t))-F_n(t-\tau_n^{\star})]}\bigr\}
= \bigl\{A_n\geq c\,e^{\gamma [(1-F_n)(t-T(t))-F_n(T(t)-\tau_n^{\star})]}\bigr\}.
\end{align*}
Therefore,
\begin{equation}\label{thisthing}
\begin{aligned}
&\b P\Bigl(\max_{n\leq M(t)} \1{\{F_n\leq 1-\nicefrac\kappa t\}} \, Z_n(t)\geq c \,e^{\gamma (t-T(t))}\Bigr)\\
&\hspace{1cm}\leq \b P(T(t)\geq\nicefrac t 2) 
 +\sum_{n=1}^{\infty}
\b P\bigl(A_n \1{\{F_n\leq 1-\nicefrac\kappa t\}} \geq c\,e^{\gamma [(1-F_n)\nicefrac t 2 -F_n(T(t)-\tau_n^{\star})]}\bigr).
\end{aligned}
\end{equation}
In terms of $\varphi(u):=\b P(A_1\geq u)$ we have
\[
\b P\left(
A_n \1{\{F_n\leq 1-\nicefrac\kappa t\}}  \geq c\, e^{\gamma [(1-F_n)\nicefrac t 2-F_n(T(t)-\tau_n^{\star})]}\right)
=\b E\left[\1{\{F_n\leq 1-\nicefrac\kappa t\}} \varphi \left(ce^{\gamma [(1-F_n)\nicefrac t 2 - F_n(T(t)-\tau_n^{\star})]}\right)\right].
\]
By Lemma~\ref{lem:doob}, we have $\varphi(u)\leq C_0\,e^{-u/2}$ so that
\begin{align*}
&\b E\left[\1{\{F_n\leq 1-\nicefrac\kappa t\}} \varphi \left(ce^{\gamma [(1-F_n)\nicefrac t 2 - F_n(T(t)-\tau_n^{\star})]}\right)\right]\\
&\hspace{2cm}\leq C_0 \, \b E\Bigl[ \1{\{F_n\leq 1-\nicefrac\kappa t\}}
\exp\left\{-{\textstyle \frac c2}e^{\gamma [(1-F_n)\nicefrac t 2 -F_n(T(t)-\tau_n^{\star})]}\right\}\Bigr].
\end{align*}
Noting that $T(t)-\tau_n^{\star}$ is deterministic we get that the expectation on the right equals
\begin{align*}
\b E\Bigl[ \1{\{F\leq 1-\nicefrac\kappa t\}}
\exp\left\{-{\textstyle \frac c2}e^{\gamma [(1-F)\nicefrac t 2 - F (T(t)-\tau_n^{\star})]}\right\}\Bigr],
\end{align*}
where $F$ is a random variable of law $\mu$. We now fix small numbers $\delta, \varrho>0$ and note that there exists a constant~$C_\varrho$ such that $e^{-y}\leq C_\varrho y^{-\varrho}$ for all $y\geq 0$. Using this, and recalling the definition of $\tau_n^\star$, we get, for $n\geq n(t)$,
\begin{align*}
\b E\Bigl[  \1{\{\delta \leq F\leq 1-\nicefrac\kappa t\}} &
\exp\left\{-{\textstyle \frac c2}e^{\gamma [(1-F)\nicefrac t 2 - F (T(t)-\tau_n^{\star})]}\right\}\Bigr]\\
& \leq \b E\Bigl[ \1{\{F\leq 1-\nicefrac\kappa t\}} 
\exp\Big\{-{\textstyle \frac c2} \left(\sfrac{n}{n(t)}\right)^{\frac{\delta \gamma}{\lambda^\star+\eps}}e^{\gamma [(1-F)\nicefrac t 2-\eps]}\Big\}\Bigr]\\
&\leq C_\varrho \left({\textstyle \frac c2}\right)^{-\varrho} e^{\gamma \varrho\eps}  \left(\sfrac{n(t)}{n}\right)^{\varrho\frac{\gamma\delta }{\lambda^\star+\eps}}  \b E\Bigl[ \1{\{F\leq 1-\nicefrac\kappa t\}}
e^{-\varrho\gamma (1-F)\nicefrac t 2}\Bigr].
\end{align*}
We now apply Lemma \ref{lem:arg_reg_var} and get, for $n\geq n(t)$,
\begin{align}
& \b E\Bigl[ \1{\{\delta \leq F\leq 1-\nicefrac\kappa t\}}
\exp\left\{-{\textstyle \frac c2}e^{\gamma [(1-F)\nicefrac t 2 - F(T(t)-\tau_n^{\star})]}\right\}\Bigr]\leq C_\varrho \left({\textstyle \frac c2}\right)^{-\varrho} e^{\gamma \varrho\eps}  \left(\sfrac{n(t)}{n}\right)^{\varrho\frac{\gamma\delta }{\lambda^\star+\eps}}  \frac{K(\kappa, \nicefrac{\varrho\gamma }{2})}{n(t)}. \label{eq:truc1}
\end{align}
Similarly we get, for $n\leq n(t)$,
\begin{align}
& \b E\Bigl[ \1{\{\delta \leq F\leq 1-\nicefrac\kappa t\}}
\exp\left\{-{\textstyle \frac c2}e^{\gamma [(1-F)\nicefrac t 2 - F(T(t)-\tau_n^{\star})]}\right\}\Bigr]\leq C_\varrho \left({\textstyle \frac c2}\right)^{-\varrho} e^{\gamma \varrho\eps}  \left(\sfrac{n(t)}{n}\right)^{\varrho\frac{\gamma}{\lambda^\star-\eps}}  \frac{K(\kappa, \nicefrac{\varrho\gamma }{2})}{n(t)}. \label{eq:truc2}
\end{align}
Applying \eqref{eq:truc1} with $\varrho_+>\frac{\lambda^\star+\eps}{\gamma\delta}$, if $n> n(t)$, and  \eqref{eq:truc2} with $\varrho_-< 
\frac{\lambda^\star-\eps}\gamma$, if $n\leq n(t)$, we get
\begin{align*}
&\sum_{n=1}^{\infty}
\b P(A_n \1{\{F_n\leq 1-\nicefrac\kappa t\}} \geq c\,e^{\gamma [(1-F_n)\nicefrac t 2 -F_n(T(t)-\tau_n^{\star})]}\bigr)\\
&\hspace{2cm}\leq C_0
 C_{\varrho_-}  \left({\textstyle \frac c2}\right)^{-\varrho_-} e^{\gamma {\varrho_-}\eps} \, \frac{K(\kappa,  \nicefrac{{\varrho_-}\gamma }2)}{n(t)} \sum_{n=1}^{n(t)} \left(\sfrac{n(t)}{n}\right)^{\varrho_- \frac{\gamma }{\lambda^\star-\eps}}\\
&\hspace{4cm}+
 C_0 C_{\varrho_+}  \left({\textstyle \frac c2}\right)^{-\varrho_+} e^{\gamma {\varrho_+}\eps} \,\frac{K(\kappa,  \nicefrac{{\varrho_+}\gamma }2)}{n(t)} \sum_{n=n(t)+1}^{\infty} \left(\sfrac{n(t)}{n}\right)^{\varrho_+\frac{\gamma\delta }{\lambda^\star+\eps}} + C_0 \, \b P(F<\delta)\\
&\hspace{2cm}\leq
C \, \big(K(\kappa, \nicefrac{{\varrho_-}\gamma }2)+ K(\kappa, \nicefrac{{\varrho_+}\gamma }{2}) \big) + C_0 \,\b P(F<\delta),
\end{align*}
where $C$ is a constant not depending on $\kappa$ or $t$, using that both sums are bounded by a constant multiple of~$n(t)$. Recall that $\lim_{\kappa\to\infty} K(\kappa, \nicefrac{{\varrho_\pm}\gamma }{2})=0$ and $\P( F < \delta) \to 0$ as $\delta\downarrow 0$. Recalling~\eqref{thisthing}  and noting that $\P( T(t) \geq \nicefrac{t}2) \to 0$,  as $t\uparrow\infty$,  completes the proof.
\end{proof}

%

\subsection{Contribution of the families born late}

\begin{lem}\label{lem:late}
For every $\eta>0$ and $c>0$ there exists $\upsilon>1$ such that, for all sufficiently large $t$, we have
\[
\b P\Bigl(\max_{\upsilon n(t)\leq n\leq M(t) }Z_n(t)\geq c \,e^{\gamma (t-T(t))}\Bigr)
\leq \eta.
\]
\end{lem}

\begin{proof}
The proof is similar to the proof of Lemma~\ref{lem:unfit}. Let $c>0$ and  define the processes  $(Z^\star_n(t) \colon t \in \mathbb R)$, the sequence $(A_n)$,  and the numbers $\tau_n^{\star}$ as above. We have
$$\{Z_n(t)\geq c\,e^{\gamma (t-T(t))}\} \subset
\{Z_n^{\star}(t)\geq c\,e^{\gamma (t-T(t))}\}
\subset \{A_n\geq c\,e^{\gamma [(1-F_n)(t-T(t))-F_n(T(t)-\tau_n^{\star})]}\}.$$
Therefore, for any $\upsilon>1$ and $n \geq\upsilon n(t)$,
\begin{equation}\label{thisthingagain}
\begin{aligned}
\b P\Bigl(\max_{n\geq \upsilon n(t)} Z_n(t)\geq c \,e^{\gamma (t-T(t))}\Bigr)
\leq \b P(T(t)\geq\nicefrac t 2) 
 +\sum_{n=\upsilon n(t)}^{\infty}
\b P\bigl(A_n\geq c\,e^{\gamma [(1-F_n)\nicefrac t 2 -F_n(T(t)-\tau_n^{\star})]}\bigr).
\end{aligned}
\end{equation}
An argument analogous to Lemma~\ref{lem:unfit} yields, for any $\delta>0$,
\[
\b P\left( A_n  \geq c\, e^{\gamma [(1-F_n)\nicefrac t 2-F_n(T(t)-\tau_n^{\star})]}\right)
\leq C_0  \, \b E\Bigl[ 
\exp\left\{-{\textstyle \frac c2}e^{\gamma [(1-F)\nicefrac t 2 - \delta (T(t)-\tau_n^{\star})]}\right\}\Bigr] + C_0 \, \b P(F<\delta),
\]
where $F$ is a random variable of law $\mu$. We now pick \smash{$\varrho > \frac{\lambda^\star+\eps}{\gamma\delta}$}. As in  Lemma~\ref{lem:unfit} we use existence of a constant $C_\varrho$ such that $e^{-y}\leq C_\varrho y^{-\varrho}$, for all $y\geq 0$, and Lemma~\ref{lem:arg_reg_var} to get
\begin{align*}
 \b E\Bigl[ 
\exp\left\{-{\textstyle \frac c2}e^{\gamma [(1-F)\nicefrac t 2 - \delta (T(t)-\tau_n^{\star})]}\right\}\Bigr]
& \leq C_\varrho \left({\textstyle \frac c2}\right)^{-\varrho} e^{\gamma \varrho\eps}  \left(\sfrac{n(t)}{n}\right)^{\varrho\frac{\gamma\delta }{\lambda^\star+\eps}}  \frac{K(0, \nicefrac{\varrho\gamma }{2})}{n(t)}.
\end{align*}
Summing over $n\geq \upsilon n(t)$ yields
\begin{align*}
\sum_{n\geq \upsilon n(t)}
\b P(A_n \geq c\,e^{\gamma [(1-F_n)\nicefrac t 2 -F_n(T(t)-\tau_n^{\star})]}\bigr)
& \leq
 C_0 C_{\varrho}  \left({\textstyle \frac c2}\right)^{-\varrho} e^{\gamma {\varrho}\eps} \, \frac{K(0,  \nicefrac{{\varrho}\gamma }2)}{n(t)} \sum_{n=\upsilon n(t)}^{\infty} \left(\sfrac{n(t)}{n}\right)^{\varrho\frac{\gamma\delta }{\lambda^\star+\eps}} + C_0 \, \b P(F<\delta) \\
& \leq C \, \upsilon^{1- \varrho\frac{\gamma\delta }{\lambda^\star+\eps}} + C_0 \, \b P(F<\delta),
\end{align*}
where $C$ is a constant that does not depend on $t$ or $\upsilon$. Finally, using that $1- \varrho\frac{\gamma\delta }{\lambda^\star+\eps}<0$,
 recalling that   $\P( F< \delta) \to 0$, as $\delta\downarrow 0$, $\P( T(t) \geq \nicefrac{t}2) \to 0$, as $t\uparrow\infty$, and plugging this into~\eqref{thisthingagain} completes the proof.
\end{proof}

{
\subsection{Families born early are not fit enough}

The following lemma is a standard extreme value result that is included for completeness.
\smallskip

\begin{lem}\label{lem:born_early}
For all $\kappa, \eta >0$, there exists $w= w(\kappa, \eta)>0$ such that, 
for all $t$ large enough,
\[\mathbb P\Big(\Gamma_t\big([-\infty, -\log w]\times [0, \kappa] \times [0,\infty]\big) = 0\Big) \geq 1-\eta.\]
\end{lem}

\begin{proof}
We have
\begin{align*}
\mathbb P\Big(\Gamma_t\big([-\infty, -\log w]\times [0, \kappa] \times [0,\infty]\big) = 0\Big)
&= \mathbb P\big(F_n < 1-\nicefrac\kappa t,\,\forall n \text{ s.t.\ }\tau_n\leq T(t)-\log w\big)\\
&= \mathbb P\big(F_n < 1-\nicefrac\kappa t,\,\forall n \leq M(T(t)-\log w)\big),
\end{align*}
where we recall that $M(T(t)-\log w)$ is the number of families that were founded before time $T(t)-\log w$.
Thus, in view of Hypothesis~\eqref{as},
\begin{align*}
\mathbb P\Big(\Gamma_t\big([-\infty, -\log w]\times [0, \kappa] \times [0,\infty]\big) = 0\Big)
&= \mathbb E\big[\mu(0, 1-\nicefrac\kappa t)^{M(T(t)-\log w)}\big]\\
&= (1+o(1))\, \mathbb E\Big[\exp\big(-M(T(t)-\log w) (\nicefrac\kappa t)^{\alpha} \ell(\nicefrac{\kappa}{t})\big)\Big],
\end{align*}
when $t$ tends to infinity.
Note that, by Lemma~\ref{lem:approx_tau_n}, with probability tending to one, 
\[\log w = T(t) - \big(T(t)-\log w\big) 
\leq \tau_{n(t)} - \tau_{M(T(t)-\log w)} 
\leq \frac1{\lambda^\star-\varepsilon} \log \frac{n(t)}{M(T(t)-1)} + \varepsilon,\]
implying that
\[
M(T(t)-1) \leq n(t)\exp\big[-(\lambda^\star-\varepsilon)(\log w-\varepsilon)\big].
\]
Recall that, using \eqref{as} again, $n(t)\sim t^{\alpha}/\ell(\nicefrac1t)$. 
We thus get
\[\mathbb P\Big(\Gamma_t\big([-\infty, -\log w]\times [0, \kappa] \times [0,\infty]\big) = 0\Big)
\geq (1+o(1))\, \exp\big(-e^{- (\lambda^\star-\varepsilon)(\log w-\varepsilon)} \kappa^{\alpha} \ell(\nicefrac{\kappa}{t})/\ell(\nicefrac1t)\big).\]
Since $\ell$ is a slowly varying function, we have that $\ell(\nicefrac\kappa t)/\ell(\nicefrac1t) \to 1$ when $t$ tends to infinity.
In conclusion,
\[\mathbb P\Big(\Gamma_t\big([-\infty, -\log w]\times [0, \kappa] \times [0,\infty]\big) = 0\Big)
\geq (1+o(1))\, \exp\big(-e^{-(\lambda^\star-\varepsilon)(\log w-\varepsilon)} \kappa^{\alpha}\big)
\to 1,
\]
as $w\uparrow\infty$, which concludes the proof.
\end{proof}
}

\section{Proof of non-extensiveness of condensation}

\subsection{Proof of Theorem~\ref{poisson}}

Let $\eta, c>0$. By Lemma~\ref{lem:unfit} there exists $\kappa=\kappa(c,\eta)$ such that
$$\liminf_{t\to\infty}\P \big( \Gamma_t([-\infty,\infty] \times{[\kappa,\infty]} \times (c,\infty])=0\big)\geq 1-\eta.$$
By Lemma \ref{lem:late} there exists $\upsilon=\upsilon(c,\eta)>1$ such that
$$\liminf_{t\to\infty}\P \big( \Gamma_t({[\log \upsilon, \infty]} \times [0,\infty] \times (c,\infty])=0\big)\geq 1-\eta.$$
By Lemma~\ref{lem:born_early}, there exists $w= w(\kappa,\eta)>0$ such that
$$\liminf_{t\to\infty}\P \big( \Gamma_t([-\infty, -\log w] \times [0,\kappa] \times [0,\infty])=0\big)\geq 1-\eta.$$
Finally, Proposition~\ref{prop:Poisson_local_Gamma} gives that $\Gamma_t$ converges on
${(}-\infty,\log\upsilon) \times[0, \kappa) \times (c,\infty]$ to the Poisson process with intensity 
measure~$\zeta$. 
Combining these four facts and using that $\eta>0$ is arbitrarily small, we get
convergence on $[-\infty,\infty] \times[0, \infty] \times (c,\infty]$. 
As this holds for all $c>0$ the proof is complete.

\subsection{Proof of Corollary~\ref{cor:as_laws}}

$(i)$\ We fix $x>0$ and apply the vague convergence proved in Theorem~\ref{poisson} to the compact set $K:=[-\infty, +\infty]\times [0,\infty] \times [x,\infty]$.
We get, as $t\to\infty$, that
\[\sum_{n=1}^{M(t)} \indi_{K}(\tau_n-T(t), (t-\tau_n)(1-F_n), e^{-\gamma (t-T(t))}Z_n(t))
\Rightarrow \mathtt{Poisson}\left(\int_K d\zeta\right).\]
Hence
\begin{align}
\mathbb P\left(e^{-\gamma (t-T(t))}\max_{n\in\{1,\ldots, M(t)\}}{Z_n(t)} \geq x\right) 
&\to \mathbb P\left(\mathtt{Poisson}\left(\int_K d\zeta\right)\geq 1\right)= 1- \exp\left(-\int_K d\zeta\right). \label{firko}
\end{align}
Integrating out gives
\begin{align*}
\int_K d\zeta
&= \int_{-\infty}^{+\infty}\int_0^{\infty} \int_x^{\infty}
\alpha f^{\alpha-1} \lambda^{\star}e^{\lambda^{\star}s} e^{-ze^{\gamma (s+f)}} e^{\gamma (s+f)} dz\ df\ ds\\
& = \int_{0}^{\infty}e^{-w} \int_0^{\infty} \alpha f^{\alpha-1} 
\int_{-\infty}^{\frac{1}{\gamma }\log \frac{w}{x}-f} \lambda^{\star} e^{\lambda^{\star}s} ds\ df\ dw\\
& = \left(\int_{0}^{\infty}e^{-w} \left(\frac{w}{x}\right)^{\frac{\lambda^{\star}}{\gamma }} dw\right) 
\left(\int_0^{\infty} \alpha f^{\alpha-1} e^{-\lambda^{\star}f} df\right)
= \frac{\Gamma(\alpha+1)\Gamma(1+\frac{\lambda^{\star}}{\gamma })}{(\lambda^{\star})^{\alpha}} \ x^{-\frac{\lambda^{\star}}{\gamma }}.
\end{align*}
Thus, the right hand side in \eqref{firko} is $1-\exp(-\Lambda x^{-\eta}),$ for
$\Lambda = \frac{\Gamma(\alpha+1)\Gamma(1+\frac{\lambda^{\star}}{\gamma })}{(\lambda^{\star})^{\alpha}}$ and $\eta = \frac{\lambda^{\star}}{\gamma }$.
Summarising,
\[\mathbb P\Big(\big(e^{-\gamma (t-T(t))}\max_{n\in\{1,\ldots,M(t)\}}{Z_n(t)}\big)^{-\eta} \leq y\Big)
\to 1-\exp(-\Lambda y) = \mathbb P(\mathtt{Exp}(\Lambda) \leq y),\]
which proves the statement.

$(ii)$\ The probability that $t(1-V(t))$ is in an interval $[a,b]$, for some $0\leq a<b$, converges to
\[\int_{-\infty}^{+\infty} \int_a^b \int_0^{\infty} e^{-\zeta([-\infty, +\infty]\times [0, \infty]\times [z, \infty])} \, \zeta(ds, df, dz),\]
where the inner integration is with respect to~$z$, the middle with respect to~$f$, and the outer with respect to $s$. We recall from above that 
\[\zeta([-\infty, +\infty]\times [0, \infty]\times [z, \infty]) = \frac{\Gamma(\alpha+1)\Gamma(1+\frac{\lambda^{\star}}{\gamma })}{(\lambda^{\star})^{\alpha}} \ z^{-\frac{\lambda^{\star}}{\gamma }}.
\]
Under~\eqref{cond}, we have  $\lambda^\star = \gamma$ and the right hand side becomes
$\frac{\alpha \Gamma (\alpha, \lambda^\star)}{z}$, where
$$\Gamma(\alpha, \lambda^\star) := 
\int_0^{\infty} f^{\alpha-1}e^{-\lambda^\star f} \,df
= \frac{\Gamma(\alpha)}{(\lambda^{\star})^{\alpha}}.$$
We get, substituting $v=e^{\gamma (s+f)}$ 
and  recalling that $\lambda^\star=\gamma$,
\begin{align*}
\int_{-\infty}^{+\infty} \int_0^{\infty} e^{-\zeta([-\infty, +\infty]\times [0, \infty]\times [z, \infty])} d\zeta(s, f, z)
&= \alpha f^{\alpha-1} e^{-\lambda^\star f} df 
\int_0^{\infty}  \bigg( \int_0^{\infty} 
v e^{-zv} dv  \bigg)\, e^{-\nicefrac{\alpha \Gamma(\alpha, \lambda^\star)}{z}} dz\\
&= \alpha f^{\alpha-1} e^{-\lambda^\star f} df 
\int_0^{\infty} 
\frac{e^{-\nicefrac{\alpha \Gamma(\alpha, \lambda^\star)}{z}}}{z^2} \, dz
= \frac{f^{\alpha-1} e^{-\lambda^\star f} df }{\Gamma(\alpha, \lambda^\star)}.
\end{align*}
$(iii)$\ By Theorem~\ref{poisson} the random variable $S(t)-T(t)$ converges to a random variable $U$ with density
\[\int_{0}^{\infty} \int_0^{\infty} e^{-\zeta([-\infty, +\infty]\times [0, +\infty]\times [z, +\infty])} \, \zeta(s, df, dz).\]
\pagebreak[3]

\subsection{Proof of Theorem~\ref{nowinner}}

We have in view of Theorem~\ref{poisson}, for all $\eps >0$ as $t\uparrow\infty$,
\[e^{-\gamma (t-T(t))} \sum_{n=1}^{M(t)} Z_n(t) 
= \int z \, d\Gamma_t(s,f,z)
\geq \int z \1_{(\varepsilon, 1)}(z) \, d\Gamma_t(s,f,z)
\to \int z \1_{(\varepsilon,1)}(z) \, d\mathtt{PPP}_\zeta(s,f,z),\]
where $\mathtt{PPP}_\zeta$ is the counting measure of a Poisson process with intensity 
measure~$\zeta$.
Observe that, for all $m\in\mathbb N$,
\[
\int z \1_{(\varepsilon, 1)}(z) \, d\mathtt{PPP}_\zeta(s,f,z)
\geq \sum_{k=0}^{m(1-\varepsilon)-1} \int_{\varepsilon+\frac k m}^{\varepsilon +\frac{k+1}m} z \, d\mathtt{PPP}_{\zeta}(s,f,z)
\geq \sum_{k=0}^{m(1-\varepsilon)-1} \left(\varepsilon+\sfrac k m\right) P_k,
\]
where $(P_k)_{k\geq 0}$ is a sequence of independent Poisson random variables of  parameters 
$\zeta(\mathbb R \times (0, \infty) \times (\varepsilon+\frac k m, \varepsilon+\frac {k+1} m))$.
As before, we find
\begin{align*}
\zeta\left(\mathbb R \times (0, \infty) \times \left(a,b\right)\right) 
&= \sfrac{\Gamma(\alpha+1)\Gamma(1+\frac{\lambda^{\star}}{\gamma })}{(\lambda^{\star})^{\alpha}} \ \Big( a^{-\frac{\lambda^{\star}}{\gamma }} - b^{-\frac{\lambda^{\star}}{\gamma }} \Big). 
\end{align*}
We abbreviate $\Lambda:=\frac{\Gamma(\alpha+1)\Gamma(1+\frac{\lambda^{\star}}{\gamma })}{(\lambda^{\star})^{\alpha}}$ and get 
\begin{align*}
\mathbb E\left[\sum_{k=0}^{m(1-\varepsilon)-1} \left(\varepsilon+\sfrac k m\right) P_k\right]
& = \Lambda \sum_{k=0}^{m(1-\varepsilon)-1}\left(\varepsilon+\sfrac k m\right) 
\left( \big(\varepsilon+\sfrac k m)^{-\frac{\lambda^\star}{\gamma}} - 
\big(\varepsilon+\sfrac {k+1} m \big)^{-\frac{\lambda^\star}\gamma} \right)\\
& \sim \frac{\Lambda}{m} \sum_{k=0}^{m(1-\varepsilon)-1} \big(\varepsilon + \sfrac{k+1}{m}\big)^{-\frac{\lambda^\star}{\gamma}}
\sim \Lambda \, \int_0^{1-\varepsilon}  \big( \varepsilon+x\big)^{-\frac{\lambda^\star}{\gamma}}\, dx, 
\end{align*}
as $m\to\infty$, by Riemann integration. The right hand side is of order $\log (\nicefrac1\varepsilon)$
if $\lambda^\star=\gamma$, and of order $\varepsilon^{1-\frac{\lambda^\star}{\gamma}}$ if $\lambda^\star>\gamma$. In any case, the expectation goes to infinity, as $\varepsilon\downarrow 0$.
With a similar reasoning we get
\[
\Var \left[\sum_{k=0}^{m(1-\varepsilon)-1} \left(\varepsilon+\frac k m\right) P_k\right]
\leq \frac{\Lambda}{m} \sum_{k=0}^{m(1-\varepsilon)-1} \frac{(\varepsilon+\frac k m)^{2-\frac{\lambda^\star}{\gamma}}}{\varepsilon+\frac{k+1}{m}}
\sim \Lambda \int_0^{1-\varepsilon} \big(\varepsilon+x\big)^{1-\frac{\lambda^\star}{\gamma}}\, dx.
\]
If~$\lambda^\star=\gamma$ the variance is therefore bounded, and otherwise it grows of a slower order than the
square of the expectation. Thus, by Chebyshev's inequality, we get
\[ \lim_{\varepsilon\downarrow 0} \lim_{m\to\infty}
\sum_{k=0}^{m(1-\varepsilon)-1} \left(\varepsilon+\frac k m\right) P_k =\infty,\]
in probability,  and this implies the claimed result.

\section{Open problems}

\subsubsection*{Precise growth of the system}
\ \\[-5mm]
A question that remains open is about the precise growth of $N(t)$ in the condensation phase.
Recall from Remark~1 that if condensation is absent we have $\log N(t)= \lambda^*t+ O(1)$ but we do not have a similarly strong statement in the condensation case. We get a lower bound on the growth by considering the size of the largest family. Under \eqref{as} this gives $\log N(t)- \gamma t + 
\gamma T(t) \to \infty$ with $T(t)\sim \nicefrac{\alpha}{\lambda^*} \log t$. A plausible conjecture would be that in the  condensation case this bound is sharp to logarithmic order, i.e.\ $\log N(t)= \lambda^*t -\alpha \log t +o(\log t).$

\subsubsection*{Shape of the condensate}
\ \\[-5mm]
Our results offer only a partial answer to the question raised in Borgs et al.~\cite{BC07} how the links in the network 
are distributed among the highest  fitnesses present in the system at any given time. The most  important question that remains open here is 
whether the families that together form the condensate have a characteristic collective behaviour prior to condensation.   The work on Kingman's model 
in Dereich and M\"orters~\cite{DM13}, and on a growth model without self-organisation in Dereich~\cite{De14}, suggests that this is indeed the case. 
In particular in the model of~\cite{De14} it is shown that for parameters chosen in the condensation 
regime, the random mass distribution in a suitably shrinking neighbourhood of the maximal fitness value
satisfies a law of large numbers with limiting shape given by a Gamma distribution.
We believe that this is a phenomenon of universal nature and conjecture the same behaviour in our model.
\bigskip

\begin{cj}[Condensation wave]
Under assumption (\ref{as}) we have
$$\lim_{t\to\infty} \Xi_t(1-\sfrac{x}{t},1)=\frac{\omega(\beta,\gamma)}{\Gamma(\alpha+1)}\int_0^x y^{\alpha} e^{-y} \, dy,$$
in probability, i.e.\ the condensation wave has the shape of a \emph{Gamma distribution} with shape parameter~$1+\alpha$.
\end{cj}

\subsubsection*{Other classes of bounded fitness distributions}
\ \\[-5mm]
In this paper we have investigated the class of fitness distributions in the maximal domain of attraction of the Weibull distribution, i.e.\ those bounded distributions of regular variation at the maximal fitness value. It would also be interesting to discuss fitness distributions with a faster decay at the maximal fitness value, for example in the maximal domain of attraction of the Gumbel distribution. This includes the interesting example of distributions with
$\log \mu(1-\eps,1) \sim -\eps^{-\gamma}$, for some $\gamma>0$. What is the shape of the condensation wave in this case? Will we also see non-extensive condensation? More generally, can we find bounded fitness distributions where we experience condensation by macroscopic occupancy? This circle of problems is currently under investigation.

\bigskip

{\bf Acknowledgements:} We are grateful to EPSRC for support through grant EP/K016075/1, and SFB~878 as well as the 
\emph{Institut f\"ur mathematische Statistik} at the University of M\"unster
for hosting PM during a sabbatical in 2015/16.   We also thank Achim Klenke for useful discussions on branching processes  with selection and mutation, which were the starting point of this project. Thanks are also due to Anna Senkevich for providing insightful simulations, and in particular the figure in this paper.
Last but not least we would like to thank the organisers of the workshop \emph{Interplay of Probability and Analysis in Applied Mathematics} at the Mathematisches Forschungsinstitut Oberwolfach in~2015, where this work was presented and extensively discussed. 
\pagebreak[3]

\end{document}